\documentclass[10pt,a4paper,3p]{elsarticle}

\usepackage{graphicx,amssymb,amsthm, amsmath}
\usepackage[IL2]{fontenc}
\usepackage{enumitem}
\usepackage{float}

\newtheorem{theorem}{Theorem}

\newcommand{\subqed}{{\ifmmode q.e.d. \else\unskip\nobreak\hfil
\penalty50\quad\null\nobreak\hfill $\blacksquare$ \parfillskip=0pt
\finalhyphendemerits=0\par\fi}}

\newcommand{\ml}{l\kern-0.55mm\char39\kern-0.3mm}

\setlist[description]{font=\normalfont\space}

\begin{document}

\begin{frontmatter}

\title{Brooks-type theorem for $r$-hued coloring of graphs}


\author{Stanislav Jendro{\ml}}
\ead{stanislav.jendrol@upjs.sk}

\author{Alfr\'ed Onderko}
\ead{alfred.onderko@student.upjs.sk}

\address{Institute of Mathematics, P.~J.~\v Saf\' arik University, Jesenn\'a 5, 040 01 Ko\v sice, Slovakia}

\begin{abstract}
An $r$-hued coloring of a simple graph $G$ is a proper coloring of its vertices such that every vertex $v$ is adjacent to at least $\min\{r, \deg(v)\}$ differently colored vertices. The minimum number of colors needed for an $r$-hued coloring of a graph $G$, the $r$-hued chromatic number, is denoted by $\chi_{r}(G)$. 
In this note we show that $$\chi_r(G) \leq (r - 1)(\Delta(G) + 1) + 2,$$
for every simple graph $G$ and every $r \geq 2$, which in the case when $r < \Delta(G)$ improves the presently known $\Delta(G)$-based upper bound on $\chi_r(G)$, namely $r \Delta(G) + 1$.

We also discuss the existence of graphs whose $r$-hued chromatic number is close to $(r-1)(\Delta + 1 ) + 2$ and we prove that there is a bipartite graph of maximum degree $\Delta$ whose $r$-hued chromatic number is $(r-1)\Delta + 1$ for every $r \in \{2, \dots, 9\}$ and infinitely many values of $\Delta \geq r + 2$; we believe that $(r-1)\Delta(G) + 1$ is the best upper bound on the $r$-hued chromatic number of any bipartite graph $G$.
\end{abstract}


\end{frontmatter}

\section{Introduction}
Throughout this note we use graph theory terminology according to the book \cite{BoMu-08}. 
However, we recall the most frequent notions.
In this paper, $G$ is a connected simple graph with vertex set $V(G)$ and edge set $E(G)$. 
The {\it{degree}} of a vertex $v$, denoted by $\deg(v)$, is the number of edges incident to $v$. 
We use $N(v)$ for the set of neighbors of $v$ in G, and $N[v]$ for the set of vertices $N(v) \cup \{v\}$.
We use $\Delta(G)$ to denote the {\it{maximum vertex degree}} of $G$. 

A {\it{$k$-coloring}} of a given graph $G$ is an assignment $\phi: V(G) \rightarrow \{1, \dots, k\}$. 
For a $k$-coloring $\phi$ and a subset of vertices $U$ we use $\phi(U)$ to denote the set of colors used on vertices of $U$, i.e., $\phi(U) = \{ \phi(u) \colon u \in U\}$.

A {\it{proper $k$-coloring}} of a graph is a vertex $k$-coloring such that any two adjacent vertices receive distinct colors. The minimum number of colors needed for a proper coloring of a graph $G$, the {\it{chromatic number}}, is denoted by $\chi(G)$. This concept is generally well known.

Wegner \cite{Weg-77} in 1977 introduced $2$-distance coloring (also known as square coloring) as a coloring in which vertices at distance at most two receive distinct colors.
The problem is to find the minimum number of colors in a $2$-distance coloring of a graph $G$, which is equivalent to find the chromatic number of the square of $G$ ($G$ with added edges between pairs of vertices at distance 2).
Because of this fact, and to prevent confusion with the notation used for $r$-hued colorings, we use $\chi(G^2)$ to denote the minimum number of colors in a $2$-distance coloring of a graph $G$.
The concept of $2$-distance colorings is also well known and broadly studied, see, for example, \cite{BoDMP-21}, \cite{CCP-19}, and \cite{MoSa-05}.

In this paper we study $r$-hued colorings of graphs.
An {\it{$r$-hued coloring}} of a simple graph $G$ is a proper coloring $\phi$ of its vertices such that
$|\phi(N(v))| \geq \min\{r, deg_G(v)\}$ for every vertex $v \in V(G)$, i.e., at least $\min \{r, \deg(v)\}$  vertices in $N(v)$ are differently colored. 
The minimum number of colors needed for an $r$-hued coloring of a graph $G$, the {\it{$r$-hued chromatic number}}, is denoted by $\chi_r(G)$.

The concept of $r$-hued colorings (also known as $r$-dynamic colorings) was introduced by Montgomery \cite{Mon-01} in 2001. There is an extensive research motivated by this concept. For a recent results and surveys see  \cite{ChFL-22}, \cite{JKSW-16}, and \cite{LMRW-18}.

In this paper, in Section 2,  we show how to color vertices of a graph $G$, one by one, in order to obtain an $r$-hued coloring of $G$ with reasonably few colors, namely $(r - 1)(\Delta(G) + 1) + 2$ colors.
To refer the state of a partially colored graph in each step, we define partial coloring and partial $r$-hued coloring of $G$.

A \textit{partial $k$-coloring} of $G$ is any mapping $\phi$ from a subset of vertices of $G$ to $\{1, \dots k\}$; similar to classical vertex colorings of $G$ (where all vertices are colored), we use $\phi(U)$ to denote the set of colors used on vertices from $U$, i.e., $\phi(U) = \{ \phi(u) \colon u \in U \cap \mathrm{dom}(\phi)\}$, where $\mathrm{dom}(\phi)$ stands for the domain of $\phi$.

A \textit{partial $r$-hued $k$-coloring} of $G$ is a partial $k$-coloring $\phi$ of $G$ such that
\begin{description}
\item[i)] if $u$ and $v$ are adjacent vertices and $u,v \in \mathrm{dom(\phi)}$, then $\phi(u) \neq \phi(v)$,
\item[ii)] for each vertex $u$ of $G$, $|\phi(N_G(u))| \geq \min\{r, |N_G(u) \cap \mathrm{dom}(\phi)|\}$.
\end{description}

In the Section 3 we use Steiner systems to show that for particular values of $r$ but infinitely many $\Delta \geq r + 2$ there are graphs with maximum degree $\Delta$ and  $r$-hued chromatic number $(r-1)\Delta + 1$.

\section{General upper bound}

One of the most famous theorems in the chromatic graph theory, Brooks' theorem \cite{Bro-41} (see also \cite{BoMu-08}), states the following:
\begin{theorem}\rm\cite{Bro-41}
    If $G$ is a simple graph of maximum degree $\Delta(G)$, then $\chi(G) \leq \Delta(G) + 1$, with equality if and only if $G$ is a complete graph or an odd cycle. 
\end{theorem}
This theorem, as well as its short proof provided  by Lovász in~\cite{Lov-75} are up to present inspiration to other authors concerning wide variety of vertex coloring problems, namely those which are based on adding another condition to a proper vertex coloring of a graph, and where the objective is to minimize the total number of colors used. 

For example, in the case of $2$-distance colorings, simple greedy algorithm shows that $\chi(G^2) \leq \Delta(G)^2 + 1$ for every simple graph $G$ (to color vertex $v$ at most $\Delta(G)$ colors used on neighbors of $v$ are forbidden to use, as well as at most $\Delta(G)(\Delta(G) - 1)$ colors used on neighbors of neighbors of $v$).
Diameter two cages such as the $5$-cycles, the Petersen graph and the Hoffman-Singleton graph (see \cite{BoMu-08}, page 84) show that there exist graphs that in fact require 
$\Delta^2 + 1$ colors for $\Delta = 2, 3$ and $7$, and possibly one for $\Delta = 57$. 

In the case of $r$-hued colorings, an easy observation shows that $\chi_1(G) \leq \chi_2(G) \leq \dots \chi_{\Delta(G)}(G) = \chi_{\Delta(G) + 1}(G)  = \dots$.
And since, evidently, $\chi_1(G) = \chi(G)$ and $\chi_{\Delta(G)}(G) = \chi(G^2)$, for every positive integer $r$ we have
$\chi(G) \leq \chi_r(G) \leq \Delta(G)^2+1$.

There is a broad research concerning $r$-hued colorings of graphs of particular families, improving the bounds on $r$-hued chromatic number and, in some cases, providing the exact value of it; see a recent survey \cite{ChFL-22}.

In this note we build upon the following result, which gives the upper bound on 
the $r$-hued chromatic number of a simple graph $G$ with the maximum degree $\Delta(G)$,
$\chi_r(G)$, for a simple graph $G$ in general. The bound in this theorem is sharp, however, only Moore graphs achieve it, and the authors of~\cite{JKSW-16} themselves ask for the best upper bound that holds for all but finitely many graphs.
\begin{theorem}{\rm\cite{JKSW-16}}\label{thm:3d10}
    Let $G$ be a graph. 
    Then $\chi_r(G) \leq r\Delta(G) + 1$,  with equality for $r \geq 2$ if and only if $G$ is regular with diameter $2$ and girth $5$.
\end{theorem}

\begin{theorem}{\rm\cite{LMP-03}}\label{thm:r=2}
    Let $G$ be a connected graph. If $\Delta(G) \geq 3$, 
    then $\chi_2(G) \leq \Delta(G) + 1$.
\end{theorem}

For some other Brooks' type bounds of $r$-hued chromatic number for specific $r$ and (or) specific families of graphs see \cite{Kar-11}, \cite{LMP-03}, and \cite{ChFL-22}. 

Our main result improves the result of Theorem~\ref{thm:3d10} in the case when $\Delta(G) > r$; as $\chi_r(G) = \chi(G^2)$ if $r \geq \Delta(G)$, the case when $\Delta(G) > r$ is the interesting one. We have:

\begin{theorem}\label{mainthm-1}
    If $G$ is a simple graph and $r \geq 2$, then $\chi_r(G) \leq (r - 1)(\Delta(G) + 1) + 2$.
\end{theorem}

\begin{proof}
    For the sake of simpler formulae, let $\Delta = \Delta(G)$ and let $A = \{1, \dots, (r - 1)(\Delta + 1) + 2\}$.

    First note that if $r > \Delta$ then $(r-1)(\Delta + 1) + 2 > r\Delta + 1$. 
    Since $\chi_r(G) \leq r \Delta + 1$ by Theorem~\ref{thm:3d10}, $\chi_r(G) \leq (r-1)(\Delta + 1) + 2$ as well.
    Hence, throughout the following, we only consider the case when $2 \leq r \leq \Delta$.
    
    We proceed in a greedy way. 
    Let $U$ be the set of uncolored vertices; $U = V(G)$ at the beginning.
    In each step, until $U$ is empty, we pick a vertex from $U$ at random, and proceed to the coloring phase.
    We distinguish two major cases, for which the following notation is used (note that $\phi$ is the partial $r$-hued coloring of $G$ at the given step):
    \begin{align*}
        S(v) &= \{ x \in N(v) \colon |\phi(N(x))| \geq r\}, \\
        W(v) &= \{ x \in N(v) \colon |\phi(N(x))| \leq r-1\}. 
    \end{align*}
    We say in the following that $u$ is a strong neighbor if $u \in S(v)$, and we say  that $u$ is a weak neighbor of $v$ if $u \in W(v)$.

    \textbf{Case 1.} If $|W(v)| \leq r-1$, we color $v$ with a color which is different from all colors of $\big( \phi(S(v)) \cup \phi(W(v)) \cup \phi(N(W(v))) \big)$.
    Since at most $r-1$ colors are used on neighbors of each weak neighbor of $v$, we have $|\phi(N(W(v)))| \leq |W(v)| (r-1) \leq (r-1)^2$.
    Moreover, $|\phi(S(v)) \cup \phi(W(v))| \leq |S(v) \cup W(v)| = \deg(v) \leq \Delta$, and thus at most $\Delta + (r-1)^2$ colors cannot be used to color $v$.
    We briefly show that $|A| \geq \Delta + (r-1)^2 + 1$, hence, there is a color to be used on $v$, which use extends $\phi$ into a partial $r$-hued coloring of $G$ with more vertices colored.
    
    We have
    $|A| - \big(\Delta + (r-1)^2 + 1\big) = \big((r - 1)(\Delta(G) + 1) + 2\big) - \big(\Delta + (r-1)^2 + 1\big) = -r^2 + (3 + \Delta)r -2-2\Delta = (\Delta + 1 - r)(r-2) \geq 0$, since $2 \leq r \leq \Delta$.
    
    We remove $v$ from $U$.
    
    \textbf{Case 2.} If $|W(v)| \geq r$, we first remove colors from colored weak vertices of $v$. 
    We then color the vertex $v$ itself, and afterwards, we color each of the weak neighbors of $v$, regardless whether the neighbor was colored or not before this step.
    To color $v$, we omit colors from $\big( \phi(S(v)) \cup \phi(N(W(v))) \big)$; by doing so we add a new color to the neighborhood of each weak neighbor, while the resulting partial coloring remains proper.
    Evidently, $|\phi(S(v)) \cup \phi(N(W(v)))| \leq \Delta (r-1)$, the worst case being the case when $\deg(v) = \Delta$ and there is exactly $r-1$ colors on neighbors of each neighbor of $v$, i.e.,  $N(v) =W(v)$.
    Since $|A| \geq 1 + \Delta (r-1)$, there is a color to be used on $v$.
    
    We now proceed into coloring vertices from $W(v)$;
    let there be a linear ordering of these vertices, say $w_1, \dots, w_k$.
    Let $j$ be an integer, $1 \leq j \leq r$. To color the vertex $w_j$, assuming all vertices $w_\ell$ for $1 \leq \ell \leq j-1$ are colored,
    we use the color from 
    \begin{align}\label{weak_neighbor_color}
       A \setminus \big( \{\phi(v)\} \cup \phi(N(w_j) \setminus \{v\}) \cup \phi(W(N(w_j) \setminus \{v\})) \cup \bigcup\limits_{\ell = 1}^{j-1} \phi(w_\ell)  \big). 
    \end{align}
    For the clarity of the proof, we count the sizes of the sets in~(\ref{weak_neighbor_color}).
    Evidently $|\{\phi(v)\}| =1$, and since $w_j$ was (at the begining of this step) a weak neighbor of $v$, $|\phi(N(w_j) \setminus \{v\})| \leq r-1$.
    Then, $|\phi(N(W(w_j) \setminus \{v\}))| \leq (\Delta - 1)(r-1)$, as the number of weak neighbors of $w_j$ different from $v$ is at most $\Delta - 1$, and each weak neighbor sees at most $r-1$ colors on its respective neighbors.
    Finally, $| \bigcup\limits_{\ell = 1}^{j-1} \phi(w_j)| \leq r-1$.
    Considering all of these sets and their sizes, at most $1 + (r-1) + (\Delta-1)(r-1) + (r-1) = 1 + (\Delta + 1)(r-1)$ colors are forbidden to use on $w_j$ for each $j \in \{1, \dots r\}$.
    Since, however, $|A| \geq 2 + (\Delta + 1)(r-1)$, there is a color to be used on each such vertex $w_j$.
    
    For each $j \in \{r+1, \dots, k\}$ we color the vertex $w_j$ with a color from
    $$ A \setminus \big( \phi(v) \cup \phi(N(w_j) \setminus \{v\}) \cup \phi(W(N(w_j)\setminus\{v\}) \big).$$
    Observe that in this case we do not have to explicitly forbid colors used on other vertices from $W(v)$ as there are at least $r$ mutually distinct colors used on neighbors of $v$, namely on $w_1, \dots, w_r$.
    Since $\big| \phi(v) \cup \phi(N(w_j) \setminus \{v\}) \cup \phi(W(N(w_j)\setminus\{v\}) \big| \leq 1 + (r-1) + (\Delta - 1)(r-1)$, there is a color to be used on $w_j$.
    Remove $v$ and all vertices from $W(v)$ from $U$.
    
    Note that, in each step, if at most $r$ neighbors of $x$ are already colored, then they are colored in a rainbow way, i.e., each of them has different color.
    Moreover, in Case 2, weak neighbors of $v$ are colored in a way that at least $r$ of them receive mutually different colors, hence, $r$-hued property holds for the vertex $v$.
    It follows from these observations that the produced coloring is in each step partial $r$-hued coloring of $G$; thus, the resulting coloring is an $r$-hued coloring of $G$.
\end{proof}

Evidently, in the case when $\Delta(G) = r$ we have $(r - 1)(\Delta(G) + 1) + 2 = r \Delta(G) +1$.
Hence, the bound in Theorem~\ref{mainthm-1} is sharp, for the same reason as the bound in Theorem~\ref{thm:3d10}, i.e., 2-distance chromatic number of a Moore graph $G$ is equal to $\Delta(G)^2 + 1$.
However, as authors of~\cite{JKSW-16} mentioned, there is only a finite number of Moore graphs, and therefore, there is a space for improvement of the currently known upper bounds after excluding some finite family of graphs such as Moore graphs.


\section{Bipartite graphs needing $(r-1)\Delta + 1$ colors}

As we mentioned before, only graphs needing $r\Delta + 1$ colors are Moore graphs.
Moreover, only examples of graphs which are known to need $r \Delta$ colors in an $r$-hued coloring are cycles whose length is not divisible by 3, in the case when $r=2$.
In~\cite{JKSW-16} a way to obtain a graph $G$ which needs $r\Delta(G) - 1$ colors is described. 
The construction involves removing some edges from several copies of an $r$-regular Moore graph and then adding new ones, resulting in an $r$-regular graph.
Even though this construction yields infinitely many such graphs, they all have their maximum degree equal to $r$, which is 2, 3 or 7 (or possibly 57 if a 57-regular Moore graph exists).
Hence, for $r \notin \{2,3,7,57\}$, the problem of finding an infinite family of graphs which would have their $r$-hued chromatic number close (up to a constant) to $r\Delta + 1$ or $(r-1)(\Delta + 1) + 2$ remains open.

In design theory, Steiner system $S(t, r, n)$ is a pair $(V, \mathcal{B})$ where $V$ is an $n$-element set and $\mathcal{B}$ is a family of $r$-element subsets of $V$, called blocks, such that each $t$-element subset of $V$ is contained in exactly one block. The necessary conditions for the existence of a Steiner system  $S(t, r, n)$ with $b$ $k$-element subsets are:
$n - 1 = k(r - 1)$ and $nk = br$.

Consider now a Levi graph~\cite{COXETER-50} (or incidence graph) of a Steiner system $(V, \mathcal{B}) = S(2,r,n)$, i.e., a bipartite graph $G$ with vertices $V \cup \mathcal{B}$ and edges between every $v \in V$ and $B \in \mathcal{B}$ if and only if $v \in B$.
As there is some block $B \in \mathcal{B}$ which contains both $v_1$ and $v_2$, for every pair of vertices from $V$, and $\deg_G(B) = r$ for every $B \in \mathcal{B}$, vertices from $V$ are colored with mutually distinct colors in any $r$-hued coloring of $G$.
Thus, $\chi_r(G) \geq |V| = n$, and, since evidently there is not any 4-cycle in $G$ (two points of $V$ are contained in exactly one block in $S(2,r,n)$), we have the following:

\begin{theorem}\label{thm_steiner}
    Let $r \geq 2$ be a positive integer. If Steiner system $S(2,r,n)$ exists, there is a biregular bipartite graph $G$ of girth at least 6 with $\chi_r(G) \geq n$. 
\end{theorem}
\begin{theorem}\label{mainthm-2}
    For every $r \in\{2, \dots, 9\}$ and infinitely many values of $\Delta \geq r + 2$, there is a graph $G$ with $\Delta(G) = \Delta$ and
    $$\chi_r(G) = (r - 1)\Delta(G) + 1.$$
\end{theorem}

\begin{proof}
    From Theorem \ref{thm_steiner} it follows that there are graphs, namely Levi graphs of Steiner systems $S(2,r,n)$, whose $r$-hued chromatic number is at least $n$.
    Note that if we were to ask the Levi graph of $(V, \mathcal{B}) = S(2,r,n)$ to be of maximum degree $\Delta$, then $n \leq \Delta(r-1) + 1$, since the number of distinct neighbors of neighbors of any vertex from $V$ is at most $\Delta (r-1)$.
    It is easy to see that a Steiner system $S(2,2,n)$ exists for every $n \geq 3$.
    In the case of $r \in \{3, 4, 5\}$ a complete characterization of Steiner systems $S(2,r,n)$ was provided by Bose~\cite{BOSE-39} and Skolem~\cite{SKOLEM-58} for $r=3$, by Hanani~\cite{Hanani-61} for $r=4$, and by Hanani~\cite{Hanani-75} for $r=5$. Namely:
    
    \begin{description}
    \item $S(2,3,n)$ exists whenever $n \equiv 1  \text{ or } 3 \mod 6$ \cite{Col-06},
    \item $S(2,4,n)$ exists whenever $n \equiv 1 \text{ or } 4 \mod 12$ \cite{AbGr-06},
    \item $S(2,5,n)$ exists whenever $n \equiv 1 \text{ or } 5 \mod 20$ \cite{AbGr-06}.
    \end{description}
    
    For $r \in \{6, 7, 8, 9\}$ sufficient conditions on existence of Steiner systems $S(2,r,n)$ are known due to Abel and Greig~\cite{AbelGreig-97} in the case when $r=6$, Janko and Tonchev~\cite{JanTon-98} in the case when $r=7$,  and due to Greig~\cite{Greig-01} in the case when $r \in \{7,8,9\}$:
    \begin{description}
    \item $S(2,6,n)$ exists if $n \geq 802$, $n \equiv 1  \text{ or } 6 \mod 15$,
    \item $S(2,7,n)$ exists if $n \geq 2606$, $n \equiv 1 \text{ or } 7 \mod 42$,
    \item $S(2,8,n)$ exists if $n \geq 3754$,
    $n \equiv 1 \text{ or } 8 \mod 56$,
    \item $S(2,9,n)$ exists if $n \geq 16498$, $n \equiv 1 \text{ or } 9 \mod 72$.
    \end{description}
     Hence, a biregular bipartite graph $G$ of girth at least 6 with $\Delta(G) =\Delta$ and 
    $\chi_r(G) \geq (r - 1)\Delta + 1$
    exists for infinitely many values of $\Delta$ and any $r \in \{3, \dots, 9\}$. 
    For example
    there is a biregular bipartite graph $G$ with $\chi_3(G) \geq 2\Delta + 1$ for every $\Delta$ congruent to $0, 1, 3$ or $4 \mod 6$;
    and there is a biregular bipartite graph $G$ with $\chi_4(G) \geq 3\Delta + 1$ for every $\Delta$ congruent to $0, 1, 4, 5, 8$, or $9 \mod 12$.

    We now show that Levi graph of any Steiner system $S(2,r, (r-1)\Delta +1)$ for $\Delta \geq r+2$ has its $r$-hued chromatic number equal to $(r-1)\Delta + 1$.
    Suppose to the contrary that there is a Levi graph $G$ of Steiner system $(V, \mathcal{B}) = S(2,r,(r-1)\Delta + 1)$ with $\chi_r(G) = k \geq (r-1)\Delta + 2$.
    As $|V| = (r - 1)\Delta(G) + 1$ and, by the above, $|\phi(V)| = (r - 1)\Delta(G) + 1$, there is an $r$-vertex $y \in \mathcal{B}$ with $\phi(y) \notin \phi(V)$.

    We show that there is a color in $\phi(V)$ that can be used to recolor $y$ in order to get an $r$-hued coloring of $G$ with less $r$-vertices having color not in $\phi(V)$. 
    We distinguish two cases. 

    \textbf{Case 1.} Let there be $\ell \geq 1$ and some vertices $z_1, \dots, z_\ell$ in $N(y)$ such that $|\phi(N(z_i) \setminus \{y\})| \geq r$ for every $i \in \{1, \dots, \ell\}$.
    Denote by $A$ the set of colors which cannot be used on $y$, i.e.,
    $$A = \{\phi(z_1), \dots, \phi(z_\ell)\} \cup \bigcup_{w \in N(y) \setminus \{z_1, \dots, z_\ell\}} \phi(N[w] \setminus \{y\}).$$
    To show that there is a color from $\phi(V)$ to be used on $y$ without violating the $r$-hued condition, it is enough that the size of $A$ is smaller than that of $\phi(V)$.
    But it certainly is, as $|A| \leq \ell + (r-\ell)r \leq 1 + (r-1)r$, hence,
    $$|\phi(V) \setminus A| \geq (r - 1)\Delta(G) + 1 - \big(1 + (r - 1)r\big) = (r - 1)(\Delta(G) - r) \geq 1.$$

    \textbf{Case 2.} Suppose that $|\phi(N(z) \setminus \{y\})| = r - 1$ for every vertex $z \in N(y)$.
    In this case, colors from $A^* =  \bigcup_{w \in N(y) \setminus \{z\}} \phi(N(w) \setminus \{y\})$ cannot be used on $y$.
    We have $|A^*| \leq r^2$ and consequently
    $$|\phi(V) \setminus A^*| \geq (r - 1)\Delta(G) + 1 - \big(r + r(r - 1)\big) = (r - 1)(\Delta(G) - r - 1) \geq 1,$$
    since $\Delta(G) = \Delta \geq r+2$.
    Thus, in this case, there is also a color from $\phi(V)$ to be used on $y$ in order to obtain an $r$-hued coloring of $G$ with less vertices colored with colors not in $\phi(V)$

    If there is no other $r$-vertex $w$ with $\phi(w) \notin \phi(V)$, the obtained $r$-hued coloring is a required $\big((r - 1)\Delta(G) + 1\big)$-coloring of $G$.
    If there is another $r$-vertex $w$ with $\phi(w) \notin \phi(V)$ we repeat the above procedure until we obtain a required $r$-hued $\big((r - 1)\Delta(G) + 1\big)$-coloring of $G$.

\end{proof}

There are other two special Steiner systems $S(2,n+1,n^2+n+1)$ and $S(2,n,n^2)$ , known as projective and affine plane of order $n$, respectively.
Using Theorem~\ref{thm_steiner}, the existence of a projective plane of order $r$ implies the existence of a bipartite biregular graph $G$ of girth at least 6 with $\Delta(G) = r$ and $\chi_r(G) \geq (r-1)\Delta + 1$, and the existence of an affine plane of order $r$ implies the existence of a bipartite biregular graph $G$ of girth at least 6 with $\Delta(G) = r+1$ and $\chi_r(G) \geq (r-1) \Delta + 1$.
It is known that if $n$ is a power of a prime, there exists a projective plane of order $n$, and therefore an affine plane of order $n$ as well (since it is a residual design to a projective plane), see~\cite{Veblen-1906}.

As mentioned before, $S(2,r,(r-1)\Delta + 1)$ does not exist if necessary conditions are not satisfied, i.e., $(r-1)\Delta \neq r(r - 1)$.
However if $(r-1)\Delta \neq r(r - 1)$, one could possibly construct a biregular bipartite graph with bipartition $X$ and $Y$ such that $\deg(x) = r$ for every $x \in X$ and $\deg(y) = \Delta$ for every $y \in Y$, which would have the property that vertices in $Y$ are colored with mutually different colors in any $r$-hued coloring, and size of $Y$ is as close to $(r-1)\Delta + 1$, as reasonably possible. Namely, we ask

\smallskip
\noindent
\textbf{Question 1.} \textit{Is it true that for every $r$ and $\Delta$, $r \leq \Delta$, there is a biregular bipartite graph with $\chi_r(G) \geq (r-1)\Delta + 1 -z$, where $z \in \{0, \dots, r-1\}$?}
\smallskip

And, concerning the exact number of colors in such graphs:

\smallskip
\noindent
\textbf{Question 2.} \textit{Is it true, that $\chi_r(G) \leq (r-1)\Delta(G) + 1$ for every bipartite graph $G$?}
\smallskip

We believe that answers to both of these questions are positive.
Note that Question 1 is equivalent to the problem of existence of some not necessarily balanced designs.
The exact number of colors in a Levi graph of any Steiner system $S(2,r,(r-1)\Delta + 1)$ is $(r-1)\Delta + 1$, which supports the possible positive answer in Question 2.
 

\bigskip
{ \bf Declaration of competing interest}

\medskip
The authors declare that they have no known competing financial interests or personal relationships that could have appeared to influence the work on this paper.

\bigskip
{\bf Acknowledgments}

\medskip
The authors are thankful to professor Alex Rosa for sharing his knowledge on problems concerning the theory of Steiner systems.

This work was supported by the Slovak Research and Development Agency under the Contract No. APVV-19-0153, by the Science Grant Agency - project VEGA 1/0574/21,
and by UPJŠ internal grant vvgs-pf-2022-2137.

\end{document}